\documentclass[11pt, a4paper, twoside]{amsart}

\usepackage{amsmath}
\usepackage{amssymb}
\usepackage{amsthm}

\usepackage{graphicx}
\usepackage{subcaption}

\usepackage{xcolor}

\usepackage{fourier}
\usepackage[T1]{fontenc}

\usepackage{mathtools}
\mathtoolsset{showonlyrefs}

\usepackage{todonotes, changes}

\allowdisplaybreaks 

\newcommand{\R}{\mathbb{R}}
\newcommand{\Z}{\mathbb{Z}}

\newcommand{\K}{\mathcal{K}}
\newcommand{\Ko}{\mathcal{K}_{os}}

\newcommand{\conv}{\mathrm{conv}}
\DeclareMathOperator{\spann}{lin} 

\DeclareMathOperator{\vol}{vol}

\newcommand{\LE}{\mathrm{G}}

\newcommand{\wdt}{\mathrm{w}}
\newcommand{\flt}{\omega}

\newcommand{\wonull}{\setminus\{0\}} 

\DeclareMathOperator{\lin}{lin}

\usepackage{hyperref}
\hypersetup{pdfpagemode=FullScreen, colorlinks=true,
  citecolor=magenta, linkcolor=cyan, urlcolor=blue}

\title[Polynomial bounds in Koldobsky's discrete slicing
problem]{Polynomial bounds in Koldobsky's discrete slicing problem}
\author{Ansgar Freyer and Martin Henk}
\address{Technische Universit\"at Wien, Institut f\"ur Diskrete Mathematik und Geometrie, Wiedner Hauptstraße 8-10/1046, A-1040 Wien}
\email{ansgar.freyer@tuwien.ac.at} 
\address{Technische Universit\"at Berlin, Institut f\"ur Mathematik, Sekr. MA4-1, Stra{\ss}e des 17 Juni 136, D-10623 Berlin}
\email{henk@math.tu-berlin.de}
\date{}
\thanks{}

\numberwithin{equation}{section}

\usepackage{mathtools}
\mathtoolsset{showonlyrefs=true}

\begin{document}

\theoremstyle{plain}
\newtheorem{theorem}{Theorem}[section]
\newtheorem{lemma}[theorem]{Lemma}
\newtheorem{corollary}[theorem]{Corollary}
\newtheorem{conjecture}{Conjecture}
\newtheorem{proposition}[theorem]{Proposition}
\newtheorem*{question}{Question}
\newtheorem{thmx}{Theorem}
\renewcommand{\thethmx}{\Alph{thmx}}
\newtheorem{lemmax}[thmx]{Lemma}

\newtheorem{theoman}{Theorem}[section]

\renewcommand{\thetheoman}{\Roman{theoman}}

\theoremstyle{definition}
\newtheorem*{definition}{Definition}

\newtheorem{remark}[theorem]{Remark}
\newtheorem{claim}{Claim}
\newtheorem*{remark*}{Remark}
\newtheorem{example}[theorem]{Example}

\begin{abstract}
 In 2013, Koldobsky posed the problem to find a constant $d_n$,
 depending only on the dimension $n$, such that for any
 origin-symmetric convex body $K\subset\R^n$ there exists an
 $(n-1)$-dimensional linear subspace $H\subset\R^n$ with
 \[
 	|K\cap\Z^n| \leq d_n\,|K\cap H\cap \Z^n|\,\vol(K)^{\frac 1n}.
      \]
In this article we show that  $d_n$ is bounded from above by
$c\,n^2\,\flt(n)/\log(n)$, where $c$ is an absolute constant and $\omega(n)$ is
the flatness constant. Due to the recent best known upper bound on
$\flt(n)$ we get a  ${c\,n^3\log(n)^2}$ bound on $d_n$.  This improves  on former bounds
  which were exponential in the dimension.        
\end{abstract}

\maketitle

\section{Introduction}
\label{sec:intro}

By a convex body we mean a non-empty convex compact set
$K\subset\R^n$. 
The class of convex bodies in $\R^n$ is denoted by $\K^n$ and the
subclass of convex bodies that are origin-symmetric is denoted by $\Ko^n$.

The classical and central slicing problem in convex geometry due to
Bourgain \cite{Bourgain1986, Bourgain1987a} asks
for the optimal constant $b_n>0$ such that for any $K\in\Ko^n$ there
exists a hyperplane $H$  such that 
\begin{equation}
\label{eq:bourgain}
\vol(K)\leq b_n \vol_{n-1}(K\cap H) \,\vol(K)^{\frac 1n}.
\end{equation}
Here $\vol(S)$ denotes the volume, i.e., $n$-dimensional Lebesgue  
measure of $S\subset\R^n$, and the $d$-dimensional volume of a set $S$
contained in a
$d$-dimensional affine plane is denoted  by $\vol_d(S)$.

It is conjectured that $b_n$ in \eqref{eq:bourgain} is an absolute constant and the current best known bound due to a very recent announced result of Klartag \cite{Klartag2023} is of order  $O(\sqrt{\log(n)})$.
This conjecture is equivalent to a multitude of other problems in
Convex Geometry and Geometric Analysis such as the isotropic constant
conjecture. It is considered to be one of the major open problems in
Convex Geometry  and for more information we refer to
\cite{BrazitikosGiannopoulosValettasEtAl2014, Chen2021, KlartagLehec2022a, MilmanPajor1989}.

Koldobsky considered generalizations of  \eqref{eq:bourgain} to
arbitrary measures (see, e.g., \cite{Koldobsky2014,
  Koldobsky2015}). For instance, in \cite{KlartagLivshyts2020,Koldobsky2014a} it is shown 
that the best-possible constant $k_n>0$ such that for any measure
$\nu$ with non-negative even continuous density on $K$  
there exists a hyperplane $H\subset\R^n$ with
\begin{equation}
\label{eq:koldobsky}
\nu(K)\leq k_n\,\nu(K\cap H)\vol(K)^{\frac 1n}
\end{equation}
is of order $O(\sqrt{n})$.   
While the measures considered in \eqref{eq:koldobsky} are continuous,
Koldobsky also asked for a discrete variant in a similar spirit. Here
the problem is to determine the best possible  constant $d_n>0$ such
that for any $K\in \Ko^n$ with $\dim (K\cap\Z^n) = n$
there exists a central hyperplane  $H\subset\R^n$,
i.e., a hyperplane passing through the origin,   with
\begin{equation}
\label{eq:discrete_slicing}
\LE(K)\leq d_n\,\LE(K\cap H)\vol(K)^{\frac 1n},
\end{equation}
where  $\LE(K) = |K\cap\Z^n|$ is the lattice point enumerator. In
\cite{AlexanderHenkZvavitch2017}  it was shown $d_n\in
O(n\,2^n)$ and  the best known lower bound is of order $\Omega(n)$
\cite[Theorem 1.6]{AlexanderHenkZvavitch2017}.
The main reason for this exponential gap is the unfortunate
circumstance that,
even though $K$ is origin-symmetric, the maximal (with respect to
lattice points) hyperplane section does not need to pass through the
origin.   
In fact, given a direction $y\neq 0$
in $\R^n$ the maximal affine hyperplane  section of $K$ orthogonal to $y$ might
contain $2^{n-1}$ times as many lattice points as the parallel section
through the origin (see \cite{AlexanderHenkZvavitch2017}).

On the other hand it is known \cite[Theorem 1.4]{freyerhenk} that for $K\in\Ko^n$   there
always exists an affine hyperplane $A$ such that
\begin{equation}
  \LE(K)^{(n-1)/n}\leq O(n)\LE(K\cap A),
  \label{eq:freyerhenk} 
\end{equation}   
and in this paper we show that there exists a (not necessarily
parallel)  central hyperplane $H$ such that $\LE(K\cap H)$ does not
deviate too much from  $\LE(K\cap A)$.  To this end we have to
distinguish between   ``large'' and ``small'' affine sections
$K\cap A$, measured with respect to the covering radius.

The covering radius in turn is related to the well-known flatness
constant $\flt(n)$, which is one of the main ingredients of our main
result. For precise definitions we refer to Section 2.  In order to
get a polynomial bound on $d_n$, we need, in particular, the following
bound on $\flt(n)$, which has been obtained very recently in \cite{ReisRothvoss}:
\begin{equation}     
\flt(n)\leq O(n\,\log(n)^3).
  \label{eq:flat} 
\end{equation}


\begin{theorem}
\label{thm:centered}
Let $K\in \K^n$, $\dim K=n$,  with centroid at the origin and
let $k\in\{1,\dots,n-1\}$. There exists a $k$-dimensional central plane $L\subset \R^n$ 
such that
\[
\LE(K)^\frac{k}{n} \leq O(\flt(n))^{n-k} O\left(\,\max\left\{\left(\frac{n+1}{k+1}\right)^{k},{\flt(k)\,k\,n}
\right\}\right) \LE(K\cap L).
\]
\end{theorem}

As a special case of our investigation and \eqref{eq:flat}, 
we obtain the desired  polynomial upper bound for $d_n$ in Koldobsky's
discrete slicing problem \eqref{eq:koldobsky}.   
\begin{corollary}
\label{cor:central}
Let $K\in \Ko^n$ with $\dim (K\cap \Z^n) = n$, $n\geq 2$. There exists 
a central
hyperplane  $H\subset \R^n$ such that
\[
\LE(K)\leq O\left(\frac{n^2\,\flt(n)}{\log(n+1)}\right) \LE(K\cap H) \vol(K)^\frac 1n.
\]
In particular,
\begin{equation}
\label{eq:main}
  \begin{split} 
\LE(K) &
\leq  O(n^3\log(n)^2) \LE(K\cap
H)\vol(K)^\frac{1}{n}.
\end{split}
\end{equation}
\end{corollary} 

It is quite likely that the right order is linear in the dimension 
which would also coincide with a result of Regev \cite{Regev2016}
where by a randomized  construction it is shown  $d_n\in O(n)$  provided
the volume of $K$ is  at most $c^{n^2}$, where $c$ is an absolute constant (see also \cite[Section 8]{GiannopoulosKoldobskyZvavitch}).

\section{Preliminaries}
\label{sec:pre}

Here we provide  further necessary concepts and results from Convex
Geometry and Geometry of Numbers which we need for our proof. For more
information we refer to the   books \cite{ArtsteinAvidanGiannopoulosMilman2015, Gardner2006,
  Gruber2007, GruberLekkerkerker1987,  Schneider2014}.  

Regarding the  volume of convex bodies, we will need two classical
inequalities. First, we make use of Kuperberg's improvemet \cite{kuperberg} of the so called reverse Blaschke-Santal\'o
inequality due to Bourgain and Milman \cite{BourgainMilman}:
\begin{equation}
\label{eq:kuperberg}
\frac{\pi^n}{n!}< \vol(K^\star)\vol(K),
\end{equation}
where $K\in\Ko^n$ and $K^\star = \{y\colon x\cdot y\leq 1,~\text{ for
  all }
x\in K\}$ denotes the polar body of $K$. The famous Mahler conjecture states
that the optimal bound is $4^n/n!$ (see, e.g., \cite{FradeliziMeyerZvavitch2023}).

Secondly, we utilize a well-known result by Rogers and Shephard
\cite{rogersshephard}  which allows us to compare the volume of $K\in\K^n$ to the volume of its difference body $K-K\in\Ko^n$:

\begin{equation}
\label{eq:rogersshephard}
\vol(K-K)\leq \binom{2n}{n}\vol(K)
.
\end{equation}
The bound is attained if and only if $K$ is a simplex, and we note  that
$\binom{2n}{n}<4^n$.

We  recall that a lattice $\Lambda\subset\R^n$ is  a
discrete subgroup of $\R^n$. If $\dim\Lambda=\dim K=d$ and $K\in\Ko^n$ 
is contained in the linear hull of $\Lambda$,
i.e, $K\subset\spann\Lambda$, then for
$1\leq i\leq d$, the $i$th successive minimum of $K$ with respect to
$\Lambda$ is given by 
\[
\lambda_i(K,\Lambda) = \min\{\lambda\geq 0\colon \dim(\lambda K \cap \Lambda) \geq i\}.
\]
The successive minima are related to the volume by Minkowski's second
theorem \cite[Theorem 23.1]{Gruber2007}:
\begin{equation}
\label{eq:minkowski}
\lambda_1(K,\Lambda)\cdots \lambda_d(K,\Lambda) \vol_d(K)\leq 2^d
\det\Lambda, 
\end{equation}
where $\det\Lambda$ is the determinant of the lattice, i.e., the $d$-dimensional volume of a
fundamental domain of the action of $\Lambda$ on $\spann\Lambda$. For
$K\in\K^n$ we denote by
\begin{equation*}
  \LE_\Lambda(K)=|K\cap\Lambda|
\end{equation*}
the lattice point enumerator of $K$ with respect to lattice points of $\Lambda$. 
In order to bound $\LE_\Lambda(K)$ of  a convex body
$K\subset\spann\Lambda$ in terms of the number of lattice points of
 lower dimensional sections
we need 
its lattice width with
respect to $\Lambda$ which is given by  
\begin{equation}
\label{eq:lattice_width}
\wdt_{\Lambda}(K) = \min_{y\in\Lambda^\star\wonull}\max_{x_1,x_2\in K} (x_1-x_2)\cdot y.
\end{equation}
Here, $\Lambda^\star = \{y\in\spann\Lambda \colon x\cdot y \in \Z
\text{ for all } x\in \Lambda\}$ is the polar lattice of $\Lambda$. In
particular, we have $(\Z^n)^\star=\Z^n$.
The lattice width can be also expressed in terms of the first successive minimum on the polar side via 
\begin{equation*}
   \wdt_{\Lambda}(K) = \lambda_1((K-K)^\star,\Lambda^\star).
\end{equation*}

A $k$-dimensional plane will be called a lattice plane if
it contains $k+1$ affinely independent points of $\Lambda$. The
orthogonal complement $L^\perp$ of a $k$-dimensional lattice plane $L$
containing the origin  
is an $(n-k)$-dimensional
lattice plane of $\Lambda^\star$. 
   The
$(\dim\Lambda-1)$-dimensional lattice planes are called lattice hyperlanes and can be parameterized
via the primitive vectors of $\Lambda^\star$, i.e., such a lattice
hypperlane $H$  is given by
\begin{equation*}
H(y,\beta)=\{x\in\spann\Lambda \colon x\cdot y = \beta\}
\end{equation*}
where  $y\in\Lambda^\star\setminus\{0\}$ is a generator of the
$1$-dimensional lattice  $\spann\{y\}\cap\Lambda^\star$ and
$\beta\in\Z$. The lattice $\Lambda$ can be decomposed as
\[
\Lambda = \bigcup_{\beta \in \Z} \left( \Lambda \cap H(y,\beta)\right)
,
\]
where none of the sections is empty. 
From \eqref{eq:lattice_width}, we can see that $\wdt_\Lambda(K)$ describes, up to a rounding, the minimum number of parallel lattice planes intersecting $K$. In particular, we have
\begin{equation}
\label{eq:Gw2}
\LE_\Lambda(K)\leq  (\wdt_{\Lambda}(K)+1) \LE_\Lambda(K \cap H(\overline
y^\star,\overline\beta)),
\end{equation}
where ${\overline y}^\star \in \lambda_1((K-K)^\star,\Lambda^\star)(K-K)^\star\cap\Lambda^\star\wonull$ and
$\overline \beta$ is chosen such that $\LE_\Lambda(K\cap H(\overline
y^\star,\beta))$ is maximized among $\beta\in\Z$. If $\dim(K\cap \Lambda)=\dim\Lambda$ we have $\wdt_\Lambda(K)\geq 1$ and, thus, \eqref{eq:Gw2} yields
\begin{equation}
\label{eq:Gw}
\LE_\Lambda(K)\leq  2\,\wdt_\Lambda(K) \LE_\Lambda(K \cap H(\overline
y^\star,\overline\beta))
\end{equation}
for such convex bodies.

The final lattice parameter that we take into account is the covering radius $\mu_\Lambda(K)$ of $K$. It is commonly defined as 
\[
\mu_\Lambda(K) = \min\{\mu\geq 0 \colon \mu K +\Lambda = \spann\Lambda\}.
\]
Due to a result of Khinchine \cite{khinchine}, there exists a constant
depending only on the dimension $d$ of $\Lambda$ that bounds the
product $\wdt_\Lambda(K)\mu_\Lambda(K)$ from above for all convex
bodies $K\subset\lin \Lambda$. The smallest number $\flt(d)$ with
\begin{equation}
\label{eq:flatness2}
\wdt_{\Lambda}(K)\mu_\Lambda(K)\leq\flt(d)
\end{equation}  
is the so-called flatness constant. To this day, the best known upper
bound on $\flt(d)$ stated in \eqref{eq:flat}  follows from a very recent result of
Reis and Rothvoss \cite{ReisRothvoss} which in turn builds upon
\cite{BanaszczykLitvakPajorEtAl1999}. On the other hand, it is easy to
see that $\flt(d)\geq d$ and the current best lower bound is due to a
recent result of Mayrhofer, Schade and Weltge \cite{weltgeetal},
showing that  $\flt(d)\geq 2d-o(d)$. Moreover, $\flt(d)$ is monotonous in $d$ as can be seen by extending a $d$-dimensional lattice $\Lambda$ to a $(d+1)$-dimensional lattice $\overline{\Lambda}$ via $\overline{\Lambda} = \Lambda \oplus \Z e_{d+1}$ and replacing $K\subseteq\lin \Lambda$ by $\overline{K} = K\times [0,\flt(d-1)]\cdot e_{d+1}$: 
it is $\wdt_{\overline{\Lambda}}(\overline{K}) = \wdt_\Lambda(K)$ and $\mu_{\overline{\Lambda}}(\overline{K}) = \mu_\Lambda(K)$.

Since the lattice width and covering radius are translation invariant, their definition and properties extend naturally to affine lattices $\Lambda + t$, where $\Lambda\subset\R^n$ is a lattice and $t\in\R^n$, together with convex bodies $K\subset t+\spann\Lambda$.

Another key ingredient of our proofs is the following result that has
been obtained recently in \cite[Proposition 1.6]{freyerlucas} (where the lower bound was proven independently by Dadush \cite{dadush}):
\begin{equation}
\label{eq:approximation}
(1-\mu_\Lambda(K))^d \frac{\vol_d(K)}{\det\Lambda}\leq \LE_\Lambda(K)\leq \frac{\vol_d(K)}{\det\Lambda}(1+\mu_\Lambda(K))^d.
\end{equation}
For the lower bound it is necessary to assume $\mu_\Lambda(K)\leq 1$. Although \eqref{eq:approximation} is stated in \cite{freyerlucas} only for the case $\Lambda=\Z^n$, the above generalization follows easily by applying a linear isomorphism that maps $\Z^d$ to $\Lambda$.

For the sake of brevity, if $\Lambda=\Z^n$, we write $\mu_\Lambda(K) =
\mu(K)$ and, likewise, $\lambda_i(K,\Lambda) =\lambda_i(K)$,
$\wdt_{\Lambda}(K) = \wdt(K)$ and $\LE_\Lambda(K)=\LE(K)$. Affine planes containing the origin will
be called central planes.

\section{Affine slices}
\label{sec:affine}

In \cite{freyerhenk} it was already shown that for any $K\in \K^n$ there exists an affine hyperplane $A\subset\R^n$ such that
\begin{equation}
\LE(K)^{\frac{n-1}{n}}\leq O(n^2) \LE(K\cap A).
\end{equation}
Here we refine this inequality by replacing the constant $O(n^2)$ with $O(\flt(n))$. More generally, we show the following.

\begin{theorem}
\label{thm:affine}
Let $K\in\K^n$, $n\geq 2$, and let $k\in\{1,\dots,n-1\}$.
There exists a $k$-dimensional affine plane $A\subset\R^n$  such that
\begin{equation}
\label{eq:affine_slicing}
\LE(K)^\frac{k}{n}\leq O(\flt(n))^{n-k}\,\LE(K\cap A).
\end{equation}
\end{theorem}

Before we come to its proof, we remark that Rabinowitz
\cite{Rabinowitz1989} settled the case $k=1$ as he showed 
 that for any convex body $K\in\K^n$ there exists a line $\ell\subset\R^n$ with
\begin{equation}
\label{eq:rabinowitz}
\LE(K)^\frac{1}{n} \leq \LE(K\cap \ell).
\end{equation}

We also need the following lemma.

\begin{lemma}
\label{lemma:hafenrundfahrt}
Let $n\geq 2$, $K\in\K^n$ with $\dim K=n$  and let
$m\in\{1,\dots,n\}$. Then we have
\[
\big(\lambda_1((K-K)^\star) \cdots  \lambda_m((K-K)^\star)\big)^n \leq \left(n!\left(\frac{8}{\pi}\right)^n\vol(K)\right)^m.
\]
\end{lemma}

\begin{proof} Let  $\lambda_i^\star =
\lambda_i((K-K)^\star)$, $1\leq i\leq n$.  As 
 \[
(\lambda_1^\star\cdots\lambda_m^\star)^n \leq (\lambda_1^\star\cdots \lambda_n^\star)^m,\]
Minkowski's second theorem \eqref{eq:minkowski} yields 
\[
(\lambda_1^\star\cdots\lambda_m^\star)^n \leq \frac{2^n}{\vol((K-K)^\star)}.
\]
The bound now follows by first applying Kuperbergs reverse Blaschke-Santal\'o inequality \eqref{eq:kuperberg} to $\vol((K-K)^\star)$, followed by the Rogers-Shephard inequality \eqref{eq:rogersshephard}.
\end{proof}

\begin{proof}[Proof of Theorem \ref{thm:affine}]
  We prove
\eqref{eq:affine_slicing} by induction on the ambient dimension $n$
and so we may assume  $\dim(K\cap \Z^n) = n$. In view of \eqref{eq:rabinowitz}, we may also assume $k\geq 2$.
We distinguish two cases depending on the lattice width of $K$.

First we assume that $\wdt(K)\leq \flt(n) + n$. Let $H=H(\overline
y^\star,\overline \beta)$ as in \eqref{eq:Gw}. Then we have  
\begin{equation}
\label{eq:1}
\LE(K) \leq O(\wdt(K))\LE(K\cap H) = O(\flt(n))\LE(K\cap H).
\end{equation}
If $k=n-1$ we are done, so we can assume $k<n-1$. By induction, there exists a $k$-dimensional affine plane $A\subset H$ with
\begin{equation}
\label{eq:2}
\LE(K\cap H)^\frac{k}{n-1} \leq O(\flt(n-1))^{n-1-k} \LE(K\cap A).
\end{equation}
Substituting \eqref{eq:2} into \eqref{eq:1} gives
\begin{equation}
  \begin{split} 
\LE(K)& \leq O(\flt(n))^{1+\frac{(n-1)(n-k-1)}{k}} \LE(K\cap
A)^\frac{n-1}{k} \\
 &\leq O(\flt(n))^{\frac{n (n-k)}{k}} \LE(K\cap
A)^\frac{n}{k},
\end{split}
\end{equation}
where we used the monotonicity of the flatness constant. Taking powers, we obtain
\begin{equation}
\begin{split}
\LE(K)^\frac{k}{n} 
&\leq O(\flt(n))^{n-k}\,\LE(K\cap A)
\end{split}
\end{equation}
as desired.

Next, we assume that $\wdt(K)\geq \flt(n) + n$. In this case,
\eqref{eq:flatness2} implies 
\begin{equation}
\label{eq:3}
\mu(K) \leq \frac{\flt(n)}{\flt(n)+n} < 1.
\end{equation}
Let $y_1,\dots,y_{n-k}\in\Z^n\setminus\{0\}$ be linearly independent
with \[y_i\in \lambda_i((K-K)^\star)\,(K-K)^\star\cap\Z^n,\, 1\leq
  i\leq n-k,\]
and let $L=\lin\{y_1,\dots,y_{n-k}\}$. Moreover, let  $\widetilde{K}=K|L$ 
the  orthogonal projection of $K$ onto $L$ and we also consider the
lattice  $\widetilde{\Lambda}=\Z^n|L$. Taking polars in the subspace $L$, we find
\begin{equation*}
(\widetilde{K}-\widetilde{K})^\star = (K-K)^\star \cap L \quad\text{and}\quad  \widetilde{\Lambda}^\star = \Z^n\cap L.
\end{equation*} 
 Hence, by the choice of $L$ we have
\[\lambda_i((\widetilde{K}-\widetilde{K})^\star,
\widetilde{\Lambda}^\star) = \lambda_i((K-K)^\star)\geq \lambda_1((K-K)^\star) = \wdt(K)\geq 1,\quad 1\leq i\leq
n-k.\]
As pointed out in  \cite[Proposition 4.1]{henkschymuraxue} it follows that 
\begin{equation}
\label{eq:4}
\LE_{\widetilde{\Lambda}} (\widetilde{K}) \leq \prod_{i=1}^{n-k}
\left(\lambda_i((\widetilde{K}-\widetilde{K})^\star,
  \widetilde{\Lambda}^\star)+1\right)\leq 2^{n-k}\prod_{i=1}^{n-k}
\lambda_i((\widetilde{K}-\widetilde{K})^\star,
\widetilde{\Lambda}^\star).
\end{equation}
Applying Lemma \ref{lemma:hafenrundfahrt} yields
\begin{equation}
\label{eq:5}
\LE_{\widetilde{\Lambda}} (\widetilde{K}) \leq  O(n)^{n-k} \vol(K)^\frac{n-k}{n}.
\end{equation}
In view of \eqref{eq:3} we can apply the volume approximation by the covering radius \eqref{eq:approximation} and obtain
\begin{equation}
\label{eq:6}
\vol(K) \leq (1-\mu(K))^{-n}\LE(K)\leq \left(\frac{\flt(n)+n}{n}\right)^n\LE(K).
\end{equation}
Combining this with \eqref{eq:5} gives
\begin{equation}
\label{eq:7}
\LE_{\widetilde{\Lambda}}(\widetilde K)\leq O(\flt(n))^{n-k} \LE(K)^\frac{n-k}{n}.
\end{equation}
Hence we obtain 
\[
\begin{split}  
\LE(K) &= \sum_{x\in \widetilde{K}\cap\widetilde{\Lambda}} \LE(K\cap
(x+L^\perp)) \\
& \leq \LE_{\widetilde{\Lambda}}(\widetilde K)
\max_{x\in\widetilde{\Lambda}} \LE(K\cap (x+L^\perp))\\
&\leq O(\flt(n))^{n-k} \LE(K)^\frac{n-k}{n}\max_{x\in\widetilde{\Lambda}} \LE(K\cap (x+L^\perp)). 
\end{split}
\]
Thus 
\[
\LE(K)^\frac{k}{n}\leq O(\flt(n))^{n-k}\, \LE(K\cap A),
\]
where $A=\widetilde{x}+L^\perp$ for some $\widetilde{x}\in\widetilde{\Lambda}$ is a $k$-dimensional lattice plane such that
$\LE(K\cap A) = \max_{x\in\widetilde{\Lambda}} \LE(K\cap (x+L^\perp))$.
\end{proof}

\section{From affine to central slices}
\label{sec:central}

For an origin-symmetric convex body $K\in\Ko^n$ the classical
concavity principle of Brunn states that for any $k$-dimensional plane $A$ 
\begin{equation}
\label{eq:brunn_vol}
\vol_k(K\cap A) \leq \vol_k(K\cap (A-A)),
\end{equation}
where $A-A$ is the central plane parallel to $A$ passing through
the origin (see, e.g., \cite[Theorem
1.2.1]{ArtsteinAvidanGiannopoulosMilman2015}). If the centroid of an $n$-dimensional
convex body $K\in\K^n$ is at the origin, i.e., we have
$0=\vol(K)^{-1}\int_K x\,\mathrm dx$, we will call the body $K$
centered. 
For those bodies 
the following analogue of
\eqref{eq:brunn_vol} has been obtained by Gr\"unbaum \cite{gru60} for $k=n-1$
and by Fradelizi \cite{fradelizi} for general $k$:
\begin{equation}
\label{eq:grunbaum}
\vol_k(K\cap A)\leq \left(\frac{n+1}{k+1}\right)^{k} \vol_k(K\cap (A-A)).
\end{equation}
For $k=n-1$, the constant in the above inequality is bounded from above
by $\mathrm{e}$. In the discrete setting, however, this factor must be
replaced by $2^{n-1}$ (even in the symmetric case) as the example
$K=\conv(\pm([0,1]^{n-1}\times \{1\})$ with $A = \{x\in\R^n\colon
x_n=1\}$ shows (see, e.g., \cite{AlexanderHenkZvavitch2017, freyerhenk}).
Nonetheless, we will show that a central plane containing ``many''
lattice points  
does still exist. 

\begin{proposition}
\label{prop:deviation}
Let $n\geq 2$,  $k\in\{1,\dots,n-1\}$, and let
$A\subset\R^n$ be a  $k$-dimensional plane.
\begin{enumerate}
\item Let $K\in \K^n$, $\dim K=n$, be centered.  Then there exists a
  $k$-dimensional central  plane  $L\subset\R^n$ such that
\[
\LE(K\cap A) \leq O\left(\max\left\{\left(\frac{n+1}{k+1}\right)^{k},{k\,n\,\flt(k)}
\right\}\right) \,\LE (K\cap L).
\]
\item Let $K\in \Ko^n$, $\dim K=n$. Then there exists a $k$-dimensional
  central plane $L\subset\R^n$ such that
\[
\LE(K\cap A) \leq O\left(\frac{\flt(k)k}{\log(k+1)}\right)\,\LE (K\cap L).
\]
\end{enumerate}
\end{proposition}
In both bounds, if $k\in O(1)$, the asymptotic order of the constant
is the same as in the corresponding continuous inequalities
\eqref{eq:brunn_vol} and \eqref{eq:grunbaum}, respectively. Moreover,
for $k=n-1$ the maximum in i) is of order $O(\flt(n)n^2)$.




\begin{proof}[Proof of Proposition \ref{prop:deviation}]
Let $\Lambda = \Z^n\cap A$.  We may assume that $\Lambda$ is a 
$k$-dimensional (affine) lattice.

For i), we first assume that
\begin{equation}
\label{eq:casedist}
\mu_\Lambda(K\cap A)\leq \tfrac{1}{(k+1)\,(n+2)}.
\end{equation}
Let $L=A-A$ and let   $\Lambda_0= \Z^n\cap L$. Then $\Lambda=t+\Lambda_0$
for some $t\in\R^n$ and as $K$ is centered we  have $(K-K)\subseteq
(n+1)K$ (cf.\ \cite[Lemma 2.3.3]{Schneider2014}). Hence,
\begin{equation}
  (K\cap A)-(K\cap A)\subseteq (n+1)(K\cap L),
\label{eq:inclusioncentered}  
\end{equation}
and as the covering radius is translation invariant and homogeneous of
degree $(-1)$ we obtain 
\begin{equation}
  \mu_{\Lambda_0}(K\cap L)\leq (n+1)\mu_\Lambda(K\cap A)<1.
\label{eq:smallmu}  
\end{equation}
Next we apply again the volume approximation via the covering radius
\eqref{eq:approximation} and together with \eqref{eq:grunbaum} and on
account of \eqref{eq:smallmu} and $\det\Lambda=\det\Lambda_0$ we obtain 
\begin{equation}
\label{eq:10}
\begin{split}
\LE(K\cap A)&\leq \frac{\vol_k(K\cap A)}{\det\Lambda}(1+\mu(K\cap A,\Lambda))^{k}\\
&\leq \left(\frac{n+1}{k+1}\right)^{k}\,
\frac{\vol_k(K\cap L)}{\det\Lambda}(1+\mu(K\cap A,\Lambda))^{k}\\
& \leq \left(\frac{n+1}{k+1}\right)^{k}\,\left(\frac{1+\mu_\Lambda(K\cap
    A)}{1-\mu_{\Lambda_0}(K\cap L)}\right)^{k}\LE(K\cap L)\\
&\leq \left(\frac{n+1}{k+1}\right)^{k}\,\left(\frac{1+\mu_\Lambda(K\cap
    A)}{1-(n+1)\mu_{\Lambda}(K\cap A)}\right)^{k}\LE(K\cap L).
\end{split}
\end{equation}

With \eqref{eq:casedist} we find
\begin{equation*}
  \begin{split}
  \left(\frac{1+\mu_\Lambda(K\cap
    A)}{1-(n+1)\mu_\Lambda(K\cap A)}\right)^{k} & \leq
\left(\frac{(k+1)(n+2)+1}{(k+1)(n+2)-(n+1)}\right)^{k}\\
&= \left(\frac{(k+1)(n+2)+1}{k(n+2)+1}\right)^{k}\leq
\left(\frac{k+1}{k}\right)^{k} \leq \mathrm{e}.
\end{split}
\end{equation*}
Combining this with \eqref{eq:10} yields
\begin{equation}
\label{eq:22}
\LE(K\cap A) \leq O\left(  \left(\frac{n+1}{k+1}\right)^{k}\right)\,\LE(K\cap L).
\end{equation}
Now suppose that \eqref{eq:casedist} does not hold. Then
\eqref{eq:flatness2} 
gives 
\begin{equation}
\label{eq:23}
\wdt(K\cap A) \leq \flt(k)\,(k+1)\,(n+2).
\end{equation}
In the special case that $\dim(K\cap A\cap \Z^n) < k$, it suffices to consider the central plane $L = \lin(K\cap A\cap\Z^n)$ whose dimension is at most $k$. For this choice of $L$ we clearly have $\LE(K\cap A) \leq \LE(K\cap L)$. We are done after extending $L$ to a linear $k$-space, if necessary.

So we can assume that $K\cap A$ contains $k$ affinely independent lattice points.
Thus, it follows from \eqref{eq:Gw} that there exists an affine
$(k-1)$-dimensional lattice plane $\widetilde{A}\subset A$ such that
\begin{equation}
\label{eq:13}
\begin{split}
\LE(K\cap A) &\leq O(\wdt(K\cap A))\LE(K\cap \widetilde{A})\\
&\leq O\left(\flt(k)\,k\,n\right)\LE(K\cap\, \spann \widetilde{A}).
\end{split}
\end{equation}
If $0\in \widetilde{A} \subset A$, then $A$ is a linear space itself
and the statement of the proposition is obvious. Otherwise, $\spann
\widetilde{A}$ is a $k$-dimensional central plane. Taking the maximum
of the upper bounds in \eqref{eq:22} and \eqref{eq:13} yields the
claim in the centered case.

For ii), we replace the threshold value in \eqref{eq:casedist}; We start by assuming
\begin{equation}
\label{eq:casedist2}
\mu_\Lambda(K\cap A) \leq \frac{\log(k+1)}{3(k+1)}.
\end{equation}
Since $K$ is symmetric, we have $\tfrac{1}{2}((K\cap A) - (K\cap A)) \subseteq K\cap L$, where again $L = A-A$ is the central plane parallel to $A$. This implies $\mu_{\Lambda_0}(K\cap L) \leq 2 \mu_\Lambda(K\cap A) $. Following the lines of \eqref{eq:10} (but using Brunn's concavity principle \eqref{eq:brunn_vol} instead of Fradelizi's bound \eqref{eq:grunbaum}) we obtain
\begin{equation}
\label{eq:41}\begin{split}
\LE(K\cap A) &\leq \left( \frac{1+\mu_\Lambda(K\cap A)}{1-2\mu_\Lambda(K\cap A)}\right)^k\LE(K\cap L) = \left( 1+\frac{3\mu_\Lambda(K\cap A)}{1-2\mu_\Lambda(K\cap A)}\right)^k \LE(K\cap L)\\
& = \left( 1+\frac{1}{k}\cdot\frac{3k}{1/\mu_\Lambda(K\cap A)-2}\right)^k\LE(K\cap L)\leq \exp(3k\mu_\Lambda(K\cap A))\LE(K\cap L),
\end{split}
\end{equation}
where we used the concavity of the function $x\mapsto \log(1+x)$ in the last step. Due to our assumption \eqref{eq:casedist2} we have 
\begin{equation}
\label{eq:42}
\LE(K\cap A) \leq (k+1)\LE(K\cap L).
\end{equation}
On the other hand, if $\mu_\Lambda(K\cap A)\geq \tfrac{\log(k+1)}{3(k+1)}$, the same argument as in i) gives
\begin{equation}
\label{eq:43}
\begin{split}
\LE(K\cap A)& \leq O(\wdt_\Lambda(K\cap A)) \LE(K\cap L^\prime)\\
&\leq O\left(\frac{\flt(k)}{\mu_\Lambda(K\cap A)}\right)\LE(K\cap L^\prime)\leq O\left( \frac{\flt(k)k}{\log(k+1)}\right)\LE(K\cap L^\prime),
\end{split}
\end{equation}
for some central $k$-plane $L^\prime\subset\R^n$. Since the bound obtained in \eqref{eq:42} is smaller than the one obtained in \eqref{eq:43}, the claim follows.
\end{proof}

From here on, the proofs of Theorem \ref{thm:affine} and Corollary \ref{cor:central} can be obtained easily by first considering a large affine slice and then estimate it against a central one using Proposition \ref{prop:deviation}.

\begin{proof}[Proof of Theorem \ref{thm:centered}]
With the help of Theorem \ref{thm:affine} we obtain an affine plane $A
\subset\R^n$ of dimension $k$ such that
\[
\LE(K)^{\frac{k}{n}}\leq O(\flt(n))^{n-k}\,\LE(K\cap A).
\]
 Proposition \ref{prop:deviation} i) yields a linear $k$-dimensional
 central plane  $L\subset\R^n$ with
\[
\LE(K\cap A) \leq O\left(\max\left\{\left(\frac{n+1}{k+1}\right)^{k},{\flt(k)\,k\,n}
\right\}\right) \LE(K\cap L).
\]
The theorem is proven.
\end{proof}

\begin{proof}[Proof of Corollary \ref{cor:central}]
Lemma \ref{lemma:hafenrundfahrt} with $m=1$ tells us that 
\begin{equation}
\wdt(K)\leq O(n)\vol(K)^\frac{1}{n}
\end{equation}
Hence,  \eqref{eq:Gw} gives the existence of  an affine hyperplane $A$ such that 
\begin{equation}
\label{eq:affine_estimate}
\LE(K)\leq O(\wdt(K))\LE (K\cap A) \leq O(n)\,\LE(K\cap A)\vol(K)^\frac{1}{n}.
\end{equation}
Applying Proposition \ref{prop:deviation} ii)  to $\LE(K\cap A)$  shows
that there exists a central hyperplane $H$ with  
\[
  \begin{split}
    \LE(K\cap A) &\leq O\left(\frac{n\flt(n)}{\log(n)}\right)\LE(K\cap
    H)\\
    &\leq O\left(n^2\log(n)^2\right)\LE(K\cap
    H),
   \end{split} 
\]
where for the last inequality we used the bound \eqref{eq:flat}. 
\end{proof}

Clearly, any strengthening of the flatness theorem \eqref{eq:flat}
directly yields an improvement of \eqref{eq:main} in Corollary
\ref{cor:central}. On the other hand, the affine estimate
\eqref{eq:affine_estimate} is sharp as the 
following example from \cite{AlexanderHenkZvavitch2017} shows:
Let
$C_n^\star =
\conv\{\pm e_1,\dots,\pm e_n\}$, where $e_i$ denotes the $i$th
standard basis vector. Then,
$C_n^\star$ contains $2n+1$ lattice points, its vertices together with the origin. Its volume is $2^n/n!$. Moreover, it is easy to check that any hyperplane section of $C_n^\star$ can contain at most $2n-1$ lattice points of $C_n^\star$. This value is attained by the coordinate sections $C_n^\star\cap e_i^\perp$. Hence,
\[
\frac{\LE(C_n^\star)}{\max_A \LE(C_n^\star\cap A) \vol(C_n^\star)^\frac{1}{n}} = O(n),
\]
where $A$ ranges over all affine hyperplanes. This shows that the linear order in \eqref{eq:affine_estimate} cannot be improved.

\section*{Acknowledgements}
We thank the anonymous referee for their helpful remarks on the paper. Ansgar Freyer is partially supported by the Austrian Science Fund (FWF) Project P34446-N.


\bibliographystyle{plain}

\end{document}